\documentclass{amsart}
\usepackage{amsmath}
\usepackage{centernot}
\usepackage{amsthm}
\usepackage{amsfonts}
\usepackage{amssymb}
\usepackage{amscd}
\usepackage{stmaryrd}
\usepackage[hidelinks]{hyperref}
\usepackage{enumitem}
\usepackage{graphicx}
\usepackage{textcomp}
\usepackage{url}

\numberwithin{equation}{section}

\theoremstyle{plain}
\newtheorem{thm}[equation]{Theorem}
\newtheorem{lemma}[equation]{Lemma}
\newtheorem{cor}[equation]{Corollary}
\newtheorem{prop}[equation]{Proposition}

\theoremstyle{definition}
\newtheorem{definition}[equation]{Definition}

\newtheorem{remark}[equation]{Remark}

\newcommand{\R}{\ensuremath \mathbb{R}}

\newcommand{\Z}{\ensuremath \mathbb{Z}}

\newcommand{\Ord}{\ensuremath \textsf{Ord}}

\begin{document}

\title[Transfinitely valued Euclidean domains]{Transfinitely valued Euclidean domains have arbitrary indecomposable order type}

\author{Chris J.\ Conidis}
\address{Department of Mathematics, College of Staten Island, The City University of New York, New York, NY 10314, USA}
\email{chris.conidis@csi.cuny.edu}

\author{Pace P.\ Nielsen}
\address{Department of Mathematics, Brigham Young University, Provo, UT 84602, USA}
\email{pace@math.byu.edu}

\author{Vandy Tombs}
\address{Department of Mathematics, Brigham Young University, Provo, UT 84602, USA}
\email{vandytombs@gmail.com}


\begin{abstract}
We prove that for every indecomposable ordinal there exists a (transfinitely valued) Euclidean domain whose minimal Euclidean norm is of that order type.  Conversely, any such norm must have indecomposable type, and so we completely characterize the norm complexity of Euclidean domains.  Modifying this construction, we also find a finitely valued Euclidean domain with no multiplicative integer valued norm.
\end{abstract}

\maketitle

\section{Introduction and history}

The class of objects with which we will concern ourselves in this paper is defined as follows.

\begin{definition}\label{Definition:TED}
An integral domain $R$ is a (transfinitely valued) \emph{Euclidean domain} if there exists a map to the class of ordinals, $\varphi:R\setminus\{0\}\to \Ord$, called a \emph{Euclidean norm}, satisfying the property that for every pair of elements $n,d\in R$ with $d\neq 0$, either $d$ divides $n$, or there exist $q,r\in R$ such that $n=qd+r$ and $\varphi(r)<\varphi(d)$.
\end{definition}

These rings generalize the finitely valued Euclidean domains, and have appeared occasionally in the literature, beginning in 1949 with Motzkin's seminal work \cite{Motzkin}.  It might seem surprising that this generalization is not studied more frequently, given that such rings naturally share nearly every property that the finitely valued Euclidean domains carry.  For instance, each Euclidean domain is a P.I.D., and the usual Euclidean algorithm terminates in finite time; both of these facts are proved using the standard arguments.

There are at least two reasons that this generalized definition is not more widely used.  First, at least \emph{a priori}, it is not obvious that this new class of Euclidean domains has any new members.  To show the containment is proper, we need some method to distinguish between the two classes.  To that end, note that the image of any Euclidean norm is contained in some fixed ordinal.  Following \cite{Clark} we define the \emph{Euclidean order type} of a Euclidean domain $R$ to be
\[
\min_{\varphi}\{\alpha\in \Ord : \varphi(R\setminus\{0\})\subseteq \alpha\}
\]
where $\varphi$ ranges among all possible Euclidean norms on $R$.  The finitely valued Euclidean domains have two different possible Euclidean order types; fields have type $\omega^{0}=1$, and non-fields have type $\omega^{1}=\omega$.  Thus, a natural question is whether any Euclidean domain has order type greater than $\omega$. Indeed, this question continues to occasionally appear in the literature, even though there are two separate examples of order type $\omega^{2}$, discovered independently by Hiblot \cite{Hiblot1,Hiblot2} and Nagata \cite{Nagata}.  These were the only known examples previous to this work.  Our first main result is a classification of all Euclidean order types, and in particular we construct Euclidean domains with arbitrarily large order type.

\begin{thm}\label{Thm:Main}
For every ordinal $\alpha$, there exists a Euclidean domain whose Euclidean order type is $\omega^{\alpha}$.  Moreover, these are the only possible Euclidean order types for Euclidean domains.
\end{thm}

A second reason that Definition \ref{Definition:TED} has not gained greater traction is  that the absolute value function on $\Z$, while being both the motivation and prototypical example of a Euclidean norm, turns out to be somewhat atypical among such norms.  As we will show in Proposition \ref{Prop:WhenTED} and Lemma \ref{LenstraLemma}, every Euclidean domain has a unique minimal Euclidean norm that is ``logarithmically superadditive'' but not multiplicative (while the opposite holds for the absolute value).  In hindsight, the naturality of the superadditive condition on the minimal norm is highlighted when studying transfinitely valued norms.

It has been a long-standing open question (see, for instance, \cite{Lemmermeyer}) whether every finitely valued Euclidean domain has some multiplicative Euclidean norm, meaning a Euclidean norm $\varphi:R\setminus\{0\}\to \Z_{\geq 0}$ such that $\varphi(xy)=\varphi(x)\varphi(y)$ for each $x,y\in R\setminus\{0\}$.  We answer this question negatively.

\begin{thm}\label{Thm:Mult}
There is a finitely valued Euclidean domain with no multiplicative Euclidean norm.
\end{thm}

We should point out that some authors posit additional conditions on Euclidean norms, or may allow defining the norm on $0$. Usually it is straightforward to translate results from one definition of norm to another, and \cite{AF} is an excellent survey on these issues.  For a broad introduction to Euclidean domains we also recommend the paper \cite{Lemmermeyer} (an updated version of that paper is also available online).  For previous work on transfinitely valued norms we recommend \cite{Motzkin,Samuel,Lenstra,Clark}, ordered by chronology.

This paper is organized as follows.  In \S\ref{Section:Ord} we review some well-known facts about ordinals that are central to the construction.  In \S\ref{Section:MinEuclid} we state an alternative, cleaner definition of Euclidean order type, and prove that such types must be indecomposable ordinals.  Next, in \S\ref{Section:Construct} we explicitly construct a Euclidean domain with Euclidean order type an arbitrary indecomposable ordinal, and in \S\ref{Section:Mult} we modify the construction to give a finitely valued Euclidean domain with no multiplicative norm.  We end the paper by giving some further consequences of the construction in \S\ref{Section:End}.

\section{Review of ordinal arithmetic}\label{Section:Ord}

All of the material in this section follows easily from the theory of ordinals developed in any standard text, such as \cite{Jech}, and so we will quickly review those parts of the theory necessary for our construction.  It is well-known, see \cite[Theorem 2.26]{Jech}, that any ordinal can be written uniquely in \emph{Cantor normal form}
\[
\omega^{\alpha_1}n_1 +\omega^{\alpha_2}n_2+\cdots + \omega^{\alpha_k}n_k=\sum_{i=1}^k\omega^{\alpha_i}n_i
\]
where we have $\alpha_1>\alpha_2>\cdots>\alpha_k$ are ordinals, the coefficients $n_1,n_2,\ldots, n_k$ are positive integers, and $k\in \omega$.  Since ordinal addition is not commutative, summation will be written from left to right (as the summation index increases), as in the equality above.

An ordinal is (additively) \emph{indecomposable} if it is nonzero and cannot be written as a sum of two smaller ordinals.  These are precisely the ordinals of the form $\omega^{\alpha}$ for some $\alpha\in \Ord$, or in other words those ordinals that have precisely one term in their Cantor normal form (with coefficient $1$).

Given two ordinals $\gamma$ and $\delta$, after adding terms with zero coefficients to their Cantor normal forms (as necessary), we may assume that the same ordinal powers appear and write $\gamma = \sum_{i=1}^{k}\omega^{\alpha_i}m_i$ and $\delta = \sum_{i=1}^{k}\omega^{\alpha_i}n_i$, where $\alpha_1>\alpha_2>\cdots>\alpha_k$ are ordinals, each $m_i$ and $n_i$ is a non-negative integer, and $k\in \omega$.  The \emph{Hessenberg sum} of $\gamma$ and $\delta$ is the ordinal
\[
\gamma\oplus \delta = \sum_{i=1}^{k}\omega^{\alpha_i}(m_i+n_i).
\]
Unlike ordinary ordinal addition, Hessenberg summation is commutative and cancellative.  We will use these facts freely.

\section{Minimal Euclidean norms}\label{Section:MinEuclid}

The Euclidean algorithm is a method of passing from one denominator to another more ``simple'' denominator.  We can measure simplicity using the divisibility relation.  For a ring $R$, and two elements $x,y\in R$, as usual we say that $y$ divides $x$, written $y|x$, if there exists some $z\in R$ such that $x=yz$.  It is now an easy task to recursively enumerate different levels of complexity with respect to this relation, as follows.

For each $\alpha\in \Ord$ we let $S_{\alpha}(R)$ be the set
\begin{equation}\label{DefiningSimplerDenoms}
\{d\in R\, :\, \text{for each } n\in R,\text{ either } d|n \text{ or there exist }  q\in R \text{ and } \beta<\alpha \text{ such that } n-qd\in S_{\beta}(R)\}.
\end{equation}
Notice that $S_0(R)$ is the set of units in $R$, and thus $S_0(R)$ is comprised of the simplest denominators (i.e., those that divide every element of $R$).  One can quickly verify that there is a containment $S_{\beta}(R)\subseteq S_{\alpha}(R)$ whenever $\beta\leq \alpha$.  Moreover, if $S_{\beta}(R)=S_{\beta+1}(R)$, then $S_{\beta}(R)=S_{\alpha}(R)$ for all $\alpha\geq \beta$.  Thus, there is a smallest ordinal $\rho(R)$ at which this sequence of subsets stabilizes, and by cardinality considerations $\rho(R)$ is smaller than the successor cardinal of the cardinality of $R$.

The following proposition collects some basic facts about these sets, and their relationship to Euclidean norms, first proved in \cite{Motzkin} and \cite{Samuel} (but also see \cite{Lenstra} and \cite{Clark}, for slightly different proofs).

\begin{prop}\label{Prop:WhenTED}
The following are equivalent for any ring $R$:
\begin{itemize}
\item[{\rm (1)}] $S_{\alpha}(R)=R$ for some $\alpha\in \textup{\textsf{Ord}}$.
\item[{\rm (2)}] $R$ is a Euclidean domain.
\end{itemize}
If these conditions hold, then it also happens that:
\begin{itemize}
\item The function $\tau:R\setminus\{0\}\to \textup{\textsf{Ord}}$ given by the rule
\[
\tau(x)=\min\{\alpha\in \textup{\textsf{Ord}} : x\in S_{\alpha}(R)\}
\]
is a Euclidean norm on $R$.
\item $\tau$ is minimal, in the sense that if $\varphi$ is any other Euclidean norm on $R$, then $\tau(x)\leq \varphi(x)$ for each $x\in R\setminus\{0\}$.
\item $\tau(R\setminus\{0\})=\rho(R)$ is the Euclidean order type of $R$, and $0\in S_{\alpha}(R)$ if and only if $\alpha\geq \rho(R)$.
\end{itemize}
\end{prop}

This last property motivates some authors to define $\varphi(0):=\varphi(R\setminus\{0\})$, for any Euclidean norm $\varphi$.  Other authors put $\varphi(0)=0$, and again other authors set $\varphi(0)=-\infty$.  We will follow the convention that $\varphi(0)$ is undefined, and similarly let the degree of the zero polynomial remain undefined.

It is well-known that the minimal Euclidean norm $\tau_{\Z}$ on $\Z$ is given by the rule $\tau_{\Z}(x)= \lfloor\log_2|x|\rfloor$.  The Euclidean algorithm corresponding to this norm is called the ``least absolute value remainder'' algorithm, as it produces remainders with minimized absolute value.  Since $\log(xy)=\log(x)+\log(y)$, we have $\tau_{\Z}(xy)\geq \tau_{\Z}(x)+\tau_{\Z}(y)$, for any $x,y\in \Z\setminus\{0\}$.  Surprisingly, \emph{all} minimal Euclidean norms satisfy a similar property.  After discovering this fact we subsequently chanced upon Proposition 3.4 in Lenstra's lecture notes \cite{Lenstra}, which established the result forty years earlier (by essentially the same proof, which we therefore do not include).  This chance encounter was quite fortuitous, as the result does not seem to have been noted elsewhere in the literature.

\begin{lemma}\label{LenstraLemma}
Let $R$ be a Euclidean domain with minimal Euclidean norm $\tau$.  If $x,y\in R\setminus\{0\}$, then
\[
\tau(xy)\geq \tau(x)\oplus \tau(y).
\]
\end{lemma}

This lemma has some strong consequences.

\begin{cor}\label{Cor:TEDindecomp}
If $R$ is a Euclidean domain, then its Euclidean order type is an indecomposable ordinal.
\end{cor}
\begin{proof}
Let $\tau$ be the minimal Euclidean norm for $R$, and write the Euclidean order type of $R$ in Cantor normal form as $\rho(R)=\omega^{\alpha_1}n_1+\omega^{\alpha_2}n_2+\cdots + \omega^{\alpha_k}n_k$.  If we suppose, by way of contradiction, that $\omega^{\alpha_1}<\rho(R)$, there then exists some element $x\in R\setminus\{0\}$ with $\tau(x)=\omega^{\alpha_1}$.  By Lemma \ref{LenstraLemma},
\[\tau(x^{n_1+1})\geq \bigoplus_{i=1}^{n_1+1}\tau(x)=\omega^{\alpha_1}(n_1+1)>\rho(R),
\]
giving us the needed contradiction.
\end{proof}

In \S\ref{Section:End} we will give some further consequences of Lenstra's result.

\section{Construction}\label{Section:Construct}

In this section we will finish the proof of Theorem \ref{Thm:Main}, by constructing a domain whose Euclidean order type is $\omega^{\alpha}$, for an arbitrary $\alpha\in \Ord$ that we now fix.  The idea of the construction is that by recursively adjoining new elements, which act as quotients, we can force the Euclidean condition, without changing the norms of old elements.

To begin, let $F$ be a field, and put $R_0=F[x_{\{\beta\},0} : 0<\beta<\omega^{\alpha}]$ where the elements $\{x_{\{\beta\},0}\}_{0<\beta<\omega^{\alpha}}$ are independent indeterminates over the field $F$.  For any element $r\in R_0\setminus F$, we define
\begin{equation}\label{Eq:SubDefinition}
{\rm Sub}(r)=\{\beta\in \Ord : \beta \text{ is an element of the first index of some variable in the support of }r\}.
\end{equation}
For instance, if $r=x_{\{5\},0}x_{\{1\},0}-x_{\{3\},0}^2$, then ${\rm Sub}(r)=\{1,3,5\}$.  (Currently all of our variables have only one ordinal in the first index, but this will change shortly.)  We next define a function $\varphi:R_0\setminus\{0\}\to\Ord$ by setting
\begin{equation}\label{Eq:VarphiDefinition}
\varphi(r)=\bigoplus_{i=1}^{n}\varphi(p_i),\text{ where }r=u\prod_{i=1}^{n}p_i\text{ is a prime factorization of $r$ with $u\in F\setminus\{0\}$},
\end{equation}
and take $\varphi(p)=\max ({\rm Sub}(p))$ for any prime $p\in R_0$.  (We also take $\varphi(u)=0$ for any $u\in F\setminus\{0\}$.)  For instance, the element $r=x_{\{5\},0}x_{\{1\},0}-x_{\{3\},0}^2$ is already prime, and thus $\varphi(r)=5$ in this case.   Note that $R_0$ is a polynomial ring over a field, hence a U.F.D., and so it makes sense to speak of prime factorizations.  The function $\varphi$ will eventually become our Euclidean norm.

Next, let
\[
S_0=\{(n,d)\in R_0\times R_0 : \gcd(n,d)=1\text{  and }\varphi(n)\geq \varphi(d)\geq 1\}.
\]
These are the pairs for which we will adjoin a new quotient $q$ such that the remainder $n-qd$ has smaller norm.  Thus, we pass to the larger ring
\[
R_1=R_0[x_{\{\beta\},1},\, y_{T,1,n,d} : 0<\beta<\omega^{\alpha},\, (n,d)\in S_0]
\]
where $T=T(n,d):={\rm Sub}(n)\cup{\rm Sub}(d)\cup\{0\}$.  For instance, if $n=x_{\{5\},0}x_{\{7\},0}^2$ and $d=x_{\{3\},0}^5-x_{\{4\},0}$ then $y_{T,1,n,d}=y_{\{0,3,4,5,7\},1,n,d}$.

We want the new variable $q:=y_{T,1,n,d}$ to act as a quotient, for the numerator $n$ and the denominator $d$.  Thus, we say that any unit multiple of the polynomial $n-qd$ is a \emph{special prime}, with corresponding \emph{special variable} $q$.   Note that the polynomial $n-qd$ is prime (as an element of $R_1$), since it has degree $1$ in the variable $q$ (and the coefficients $n$ and $d$ are relatively prime because $(n,d)\in S_0$).

Define ${\rm Sub}(r)$ for any $r\in R_1\setminus F$ exactly as in (\ref{Eq:SubDefinition}).  (Notice that each of the $y$ variables has multiple ordinals in its first subscript, which is why we phrased (\ref{Eq:SubDefinition}) as we did.)  Subsequently, define $\varphi$ on products of primes in $R_1$, using (\ref{Eq:VarphiDefinition}), and define $\varphi$ on irreducible elements $p\in R_1$ by the rule
\begin{equation}\label{Eq:VarphiPrimeSpecial}
\varphi(p)=\begin{cases}
\max\{\beta\in T : \beta<\varphi(d)\} & \text{ if }p \text{ is a special prime, with special variable $y_T$},\\
\max({\rm Sub}(p)) & \text{ otherwise}.
\end{cases}
\end{equation}
(When $p\in R_0$ this definition agrees with the one given originally, so we really have extended $\varphi$.)

We now recursively define rings $R_j$ for each $j<\omega$, and extend $\varphi$ to these rings.  The details are not much different in this general case than when passing from $R_0$ to $R_1$.  Suppose that we have defined $R_i$ (for some $i\geq 1$) and have extended $\varphi$ to this ring so that both (\ref{Eq:VarphiDefinition}) and (\ref{Eq:VarphiPrimeSpecial}) hold.  To define $R_{i+1}$, we let
\[
S_i=\{(n,d)\in R_i\times R_i : \gcd(n,d)=1\text{ and }\varphi(n)\geq \varphi(d)\geq 1\}
\]
and then we take
\[
R_{i+1}=R_i[x_{\{\beta\},i+1},\, y_{T,i+1,n,d} : 0<\beta<\omega^{\alpha},\, (n,d)\in S_{i}]
\]
where $T$ is defined (solely in terms of $n$ and $d$) as before.  Each of the variables $q=y_{T,i+1,n,d}$ is a special variable for exactly one new special prime $n-qd$ (up to unit multiples).  We extend the function ${\rm Sub}:R_{i+1}\setminus F\to \Ord$ using (\ref{Eq:SubDefinition}), and extend $\varphi$ to $R_{i+1}$ by (\ref{Eq:VarphiDefinition}) and (\ref{Eq:VarphiPrimeSpecial}).  This completes the recursive construction.  It should be noted that (\ref{Eq:VarphiPrimeSpecial}) defines $\varphi$ in terms of its earlier values, which causes no problems because if $n-qd\in R_{j+1}$ is a special prime (with special variable $q$), then $d\in R_j$, so the recursion stops in finite time.

Now, put $R_{\infty}=\bigcup_{j=0}^{\infty}R_j$ and let $U=\{r\in R_{\infty} : \varphi(r)=0\}$.  In other words, $U$ is the set of elements that are products (including the empty product) of special primes with $\varphi$-value $0$ and nonzero elements of $F$.  We put $R=U^{-1}R_{\infty}$, and we extend $\varphi$ to $R\setminus\{0\}$ in the obvious way.  (Namely, given $r\in R_{\infty}$ and $u\in U$ we let $\varphi(u^{-1}r):=\varphi(r)$.)  Now that we have defined the domain $R$, which will serve as our example, we next prove it has the necessary properties.

\begin{lemma}\label{Lemma:RisEuclidean}
With $R$ and $\varphi$ as defined above, the map $\varphi$ is a Euclidean norm on $R$.
\end{lemma}
\begin{proof}
Let $n,d\in R$ with $d\neq 0$, and assume $d\nmid n$.  Our goal is to find some $q\in R$ with $\varphi(n-qd)<\varphi(d)$.  Since $\varphi$ is invariant under multiplication by units, we may as well assume $n,d\in R_{\infty}$ and neither has a special prime factor from $U$.  If $\varphi(n)<\varphi(d)$, we can take $q=0$, so we reduce to the case $\varphi(n)\geq \varphi(d)$.  Letting $r=\gcd(n,d)$ we may write $n=n'r$ and $d=d'r$ for some $n',d'\in R_{\infty}$ with $\gcd(n',d')=1$.  Note that $d'\neq 1$ since $d\nmid n$, and thus we must have $\varphi(d')\geq 1$.  Moreover by (\ref{Eq:VarphiDefinition}) we have,
\begin{equation}\label{Eq:FirstLemma1}
\varphi(n')\geq \varphi(d').
\end{equation}
Taking $q=y_{T,i,n',d'}$ (where $i$ is chosen large enough so that $n',d'\in R_{i-1}$) then $\varphi(n'-qd')<\varphi(d')$ by the first case of (\ref{Eq:VarphiPrimeSpecial}).  Thus, by (\ref{Eq:VarphiDefinition}) we find
\begin{equation}\label{Eq:FirstLemma2}
\varphi(n-qd)=\varphi(n'-qd')\oplus \varphi(r)<\varphi(d')\oplus\varphi(r)=\varphi(d),
\end{equation}
which verifies the claim.
\end{proof}

The map $\varphi$ does more than serve as a Euclidean norm.

\begin{lemma}\label{Lemma:RhasMinEuclidNorm}
The map $\varphi$ defined above equals the minimal Euclidean norm $\tau$ on $R$.
\end{lemma}
\begin{proof}
We will show $\varphi(d)=\tau(d)$ for each $d\in R\setminus\{0\}$, by induction on $\tau(d)$.  First note that $\varphi(d)=0$ if and only if $d$ is a unit, if and only if $\tau(d)=0$.  This covers the base case.

Assume inductively that any element with $\tau$-norm smaller than $\beta\geq 1$ has the same $\varphi$-norm.  Further assume, by way of contradiction, that $\beta=\tau(d)<\varphi(d)$ for some nonzero $d\in R$.  By Lemma \ref{LenstraLemma} and (\ref{Eq:VarphiDefinition}), and from the fact that $\tau(r)\leq \varphi(r)$ for all $r\in R\setminus\{0\}$, we may reduce to the case that $d\in R_{\infty}$ is irreducible.  In particular, we have $d\in R_j$ for some $j<\omega$.
Setting $n:=x_{\{\beta\},j+1}$, then as $d\nmid n$, this means there exists some $q\in R$ such that $\tau(n-qd)<\beta$.  From our inductive assumption we then have $\varphi(n-qd)=\tau(n-qd)$.  After clearing denominators, we obtain
\begin{equation}\label{Eq:ContradictionEquality}
u n-q'd = r
\end{equation}
for some $u,q',r\in R_{\infty}$, with $u\in U$ and $\varphi(r)<\beta$.  If $u$ and $q'$ share a factor, then we can remove that factor from both sides of (\ref{Eq:ContradictionEquality}); so we may assume $u$ and $q'$ share no factor in $R_{\infty}$.  Also recall that $d$ is irreducible with positive $\varphi$-norm, so it shares no factor with $u$.

Moreover, if $n$ is a factor of $q'd$, then it is a factor of $r$, and so $\varphi(r)\geq \varphi(n)=\beta$, yielding a contradiction.  This proves that $u n$ and $q'd$ share no common factors in $R_{\infty}$, and by (\ref{Eq:ContradictionEquality}) the same fact is true for any two of the three polynomials $un$, $q'd$, and $r$.

Next, note that $n=x_{\{\beta\},j+1}$ does not appear in any monomial in the support of $d$, since $d\in R_j$.  Let $\psi:R_{\infty}\to R_{\infty}$ be the unique ring homomorphism fixing $F$ and all variables in $R_{\infty}$, except that we take $\psi(n)=0$.  Applying $\psi$ to (\ref{Eq:ContradictionEquality}) we have $-\psi(q')d=\psi(r)$, or in other words $d|\psi(r)$.  As $d\nmid r$, we have $\psi(r)\neq r$, and thus $n$ appears in some irreducible factor $r_1$ of $r$.  This factor must be special, else (\ref{Eq:VarphiDefinition}) and (\ref{Eq:VarphiPrimeSpecial}) would give
\begin{equation}\label{Eq:SpecialFactorTooBig}
\varphi(r)\geq \varphi(r_1)\geq \beta,
\end{equation}
a contradiction.  Thus, we have shown that a special variable occurs in the support of $r$, that special variable's first index contains $\beta$, and its second index is larger than $j$.  Let $q_1:=y_{T_1,k_1,n_1,d_1}$ be a special variable that occurs in the support of either $r$ or $u$, such that $\beta\in T_1$, and $k_1>j$ is maximal with respect to these properties.

Consider what happens if $q_{1}$ occurs in some irreducible factor $r'$ of $r$, but not as a corresponding special variable.  The factor $r'$ cannot be special by maximality of $k_1$, but then we get a contradiction as in (\ref{Eq:SpecialFactorTooBig}).  Therefore, if $q_{1}$ occurs in some irreducible factor of $r$, it must occur in a special prime, and only as the corresponding special variable.  On the other hand, if $q_{1}$ occurs in some irreducible factor of $u$, then since every such factor is special, the maximality of $k_1$ forces $q_{1}$ to be the corresponding special variable.  As $\gcd(u,r)=1$, we see that $q_{1}$ occurs in exactly one prime factor (not counting multiplicity) of $u$ or $r$ (but not both), and only as the corresponding special variable.  Furthermore $k_{1}>j$, and so $q_{1}$ does not appear in $d$.  The only way (\ref{Eq:ContradictionEquality}) can then hold is if $q_{1}$ also appears in $q'$.

First consider the case when $q_{1}$ occurs in $r$.  Write $r=s(n_{1}-q_{1}d_{1})^m $ for some integer $m\geq 1$, maximal with respect to $s\in R_{\infty}$.  Thus, as polynomials in the variable $q_{1}$, the right-hand side of (\ref{Eq:ContradictionEquality}) has leading coefficient $(-1)^md_{1}^{m}s$.  The only term on the left-hand side of (\ref{Eq:ContradictionEquality}) in which $q_{1}$ appears is $q'$, and thus the left-hand side of (\ref{Eq:ContradictionEquality}) has leading term divisible by $d$.  Thus $d|d_{1}^ms$.  As $\gcd(d,r)=1$ and $d$ is irreducible, we have $d|d_{1}$.  This forces $\varphi(d_{1})\geq \varphi(d)$, and hence by (\ref{Eq:VarphiPrimeSpecial}) we have
\[
\varphi(r)\geq \varphi(n_{1}-q_{1}d_{1})\geq \beta
\]
since $\beta\in T_{1}$ and $\beta<\varphi(d)$.  This contradicts the fact $\varphi(r)<\beta$.

Finally consider the case when $q_{1}$ occurs in $u$.  By the same argument as in the previous paragraph, $d|d_{1}$, and so $\varphi(d_{1})\geq \varphi(d)>\beta$.  Thus, by (\ref{Eq:VarphiPrimeSpecial}) and the fact that $\beta\in T_{1}$, we have
\[
\varphi(u)\geq \varphi(n_{1}-q_{1}d_{1})\geq \beta\geq 1,
\]
a final contradiction.
\end{proof}

Now that we know that $\varphi$ is the minimal Euclidean norm, we see that the Euclidean order type of $R$ is
\[
\{\varphi(x) : x\in R\setminus\{0\}\}.
\]
On the one hand, $\varphi(1)=0$ and $\varphi(x_{\{\beta\},0})=\beta$, so this set contains every ordinal less than $\omega^{\alpha}$.  On the other hand, any ordinal that appears in the first index of any of the variables is $<\omega^{\alpha}$.  Thus, by (\ref{Eq:VarphiDefinition}) and (\ref{Eq:VarphiPrimeSpecial}), we see that the Euclidean order type is exactly $\omega^{\alpha}$, finishing the proof of Theorem \ref{Thm:Main}.

We should point out that this construction does more than create a ring with Euclidean order type $\omega^{\alpha}$.  The minimal Euclidean algorithm for this ring is completely effective; given any pair $n,d\in R$ with $d\neq 0$ and $d\nmid n$, the proof of Lemma \ref{Lemma:RisEuclidean} tells us explicitly how to find a quotient $q$ such that the remainder $n-qd$ has smaller norm, assuming we can factor polynomials over $F$.  (Taking $F$ to be a finite field makes factoring of multivariate polynomials completely effective for trivial counting reasons, but of course there are much better algorithms available.)

\section{Euclidean domains and multiplicative norms}\label{Section:Mult}

In this section our goal will be to prove Theorem \ref{Thm:Mult}, that there is a finitely valued Euclidean domain with no multiplicative Euclidean norm.  We follow the construction used in the preceding section, with only the following five changes.

(1) We fix $\alpha=1$, so that the ring will be a finitely valued Euclidean domain.

(2) We restrict $F$ so that it is a field of characteristic $0$.

(3) At the zeroth stage of the construction, we adjoin one more variable $z=z_{\{1\},0}$.

(4) The biggest change is in the definition of $\varphi$.  Instead of (\ref{Eq:VarphiDefinition}) we employ:
\begin{equation}\label{Eq:ModifiedVarphi}
\varphi(r)=\varphi(z^k)\oplus \bigoplus_{i=1}^{n}\varphi(p_i), \text{ where } r=uz^k\prod_{i=1}^n p_i \text{ is a prime factorization, $u\in F\setminus\{0\}$, and }z\nmid p_i,
\end{equation}
and we take $\varphi(z^k)=k^k$ for each integer $k\geq 1$ (as well as $\varphi(z^0)=0$).  Thus, $\varphi$ still satisfies (\ref{Eq:VarphiDefinition}) except on powers of $z$.  We leave the definition of $\varphi$ unchanged on the other primes.

(5) At the recursive stages of the construction we expand the set $S_i$ by allowing pairs of (nonzero) elements $(n,d)\in R_i\times R_i$ still satisfying $\gcd(n,d)=1$ but with $\varphi(n)<\varphi(d)$ if $z|n$.  This produces some new special primes and special variables, which are subject to the same conditions as before.

With these modification in place, Lemma \ref{Lemma:RisEuclidean} still holds, with the following minor changes to the proof.  (In this paragraph, all notations follow those introduced in that proof.)  No change is needed when (\ref{Eq:FirstLemma1}) holds, so we only need to consider the situation when $\varphi(n)\geq \varphi(d)$ and yet $\varphi(n')<\varphi(d')$.  That can only happen if $z|r$ and $z|n'$, since $\varphi(z^{i+j})\geq \varphi(z^i)+\varphi(z^j)$ for any integers $i,j\geq 0$.  In that case there is still a special variable $q=y_{T,i,n',d'}$ due to the changes we made to the set $S_i$.  Moreover, $z\nmid d'$ (since $\gcd(n',d')=1$) and hence $z\nmid (n'-qd')$.  Thus, by (\ref{Eq:ModifiedVarphi}) the equation (\ref{Eq:FirstLemma2}) is unchanged and the lemma is complete.

Next, we will show that Lemma \ref{Lemma:RhasMinEuclidNorm} is still valid.  (Again, in this paragraph, all notations follow those introduced in that lemma.)  In the second paragraph of the proof of that lemma, when we attempt to reduce to the case where $d$ is irreducible by using (\ref{Eq:VarphiDefinition}), we must now instead use (\ref{Eq:ModifiedVarphi}), and so there is also the possibility that $d=z^k$ for some $k\geq 2$.  We only need to deal with this new case, as the proof in the original case may be left unchanged.  The proof proceeds as before until we reach the point where $d|d_{1}^m$.  (Further, we will only handle the case where $q_1$ occurs in $r$, as the case when $q_1$ occurs in $u$ is similar.)  From $d|d_1^m$ we can only conclude $z|d_1$ (at least initially).  View (\ref{Eq:ContradictionEquality}) once again as a polynomial equation in the variable $q_1$.  The coefficient of $q_1$ (i.e., the coefficient of the degree $1$ term) on the right-hand side is $-mn_{1}^{m-1}d_1s$.  The coefficient of $q_1$ on the left-hand side of (\ref{Eq:ContradictionEquality}) is again divisible by $d$.  Thus, since $m\neq 0$ (using the fact that $F$ has characteristic $0$) we have $d|(n_1^{m-1}d_1s)$.  Since $z|d_1$ and $\gcd(n_1,d_1)=1$, we must have $d|d_1$.  The remainder of the proof is unchanged.

Now that we know $\varphi$ is the minimal Euclidean norm, we are ready to prove that there is no multiplicative Euclidean norm on $R$.  Let $\psi:R\setminus\{0\}\to \Z_{\geq 0}$ be any Euclidean norm on $R$ with codomain $\Z_{\geq 0}$.  Set $k:=\psi(z)$, and notice that since $z$ is not a unit, we must have $k\geq 1$.  Fix $\ell:=k+1$, so that $k^{\ell}<\ell^{\ell}$.  If $\psi$ were multiplicative, we then would have
\[
\psi(z^{\ell})=\psi(z)^{\ell}=k^{\ell}<\ell^{\ell}=\varphi(z^{\ell}),
\]
contradicting the fact that $\varphi$ is the minimal Euclidean norm.

\begin{remark}
Lenstra in \cite{Lenstra} uses a slightly different definition of multiplicative norm, by allowing values in the real numbers.  To be precise, let us say that a map $\psi:R\setminus\{0\}\to \R$ is a \emph{real-multiplicative norm} if the following three conditions hold:
\begin{itemize}
\item[(1)] the image of $\psi$ is well-ordered (under the usual ordering on $\R$),
\item[(2)] for each $n,d\in R$ with $d\neq 0$ either $d|n$ or there exists some $q\in R$ with $\psi(n-qd)<\psi(d)$, and
\item[(3)] for each $x,y\in R\setminus\{0\}$ we have $\psi(xy)=\psi(x)\psi(y)$.
\end{itemize}

This notion is more inclusive, as it allows for more general order types in the image.  There are many multiplicative submonoids of $\R$ that are well-ordered but not of order type $\omega$.

We were unable to prove that there is a finitely valued Euclidean domain without a real-multiplicative norm.  Even more generally, one might ask: If $R$ is a Euclidean domain, then is there is some Euclidean norm $\psi:R\setminus\{0\}\to \Ord$ and some binary operation $\circledast$ on $\Ord$, such that under the canonical ordering $<$ of $\Ord$ the structure $(\Ord,\circledast,<)$ is a commutative, cancellative, strictly well-ordered monoid, and $\psi(xy)=\psi(x)\circledast \psi(y)$ for each $x,y\in R\setminus\{0\}$?

The ring we constructed in this section has such a norm, taking $\circledast$ to be the Hessenberg sum operation $\oplus$ on $\Ord$, and taking $\psi=\varphi$ on prime powers except that $\psi(z^k)=\omega k$.
\end{remark}

\begin{remark}
Lenstra observed in \cite{Lenstra}, on page 13, that for all (then) known finitely valued Euclidean domains $R$ with minimal Euclidean norm $\tau$, there exists some $k\in \Z_{\geq 0}$ such that
\[
\tau(x)+\tau(y)\leq \tau(xy)\leq \tau(x)+\tau(y)+k
\]
for all $x,y\in R\setminus\{0\}$.  The ring we constructed in this section fails this condition by taking $x=y=z^{\ell}$ (for sufficiently large $\ell$, depending on $k$).
\end{remark}

\section{Additional consequences}\label{Section:End}

A number of interesting problems are posed in Section 5 of the survey article \cite{Clark}.  Question 5.4 asks whether there are Euclidean domains $R_1$ and $R_2$ such that (in our notation) $\rho(R_1)\geq \rho(R_2)$ but $\rho(R_1)+\rho(R_2)<\rho(R_1)\oplus \rho(R_2)$.  By Corollary \ref{Cor:TEDindecomp} we see that the answer is negative, since only indecomposable ordinal types are possible.

The methods of our paper also completely answer Questions 5.6 of \cite{Clark}, which asks what ordinals $\rho(R)$ are possible when $R$ is any commutative ring.  Indeed, our work classifies $\rho(R)$ for all Euclidean \emph{rings}, i.e., those commutative rings satisfying the Euclidean condition but not necessarily domains.  First, any such ring is a finite direct product of Euclidean domains and special Artinian principal rings by work found in \cite{Lenstra} or \cite{Clark}.  Lenstra, in \cite[Corollary 2.3]{Lenstra}, classifies the minimal Euclidean norm of such a product as the Hessenberg sum of the minimal norms on each factor (which, coincidentally, also completely answers Question 5.3 of \cite{Clark}).  We then see that $\rho(R)$ can equal \emph{any} ordinal, even restricting $R$ to be a finite direct product of Euclidean domains.

Pete Clark, in a personal communication, pointed out a nice further application of \cite[Corollary 2.3]{Lenstra}.  Namely, Lemma \ref{LenstraLemma} is true if we work in a Euclidean ring, but restrict $x$ and $y$ to be non-zero-divisors.  (Here is the quick sketch: View $R$ as a direct product of certain rings, as in the previous paragraph.  Being a non-zero-divisor in a special Artinian principal ring means it is a unit, having minimal norm $0$.  Now apply Corollary 2.3 from \cite{Lenstra}.)

\section{Acknowledgements}

The first author would like to thank Rod Downey and Asher Kach for introducing him to the problems answered in this paper.  We also thank Pete Clark for comments on an earlier version of the paper, and the anonymous referee for valuable comments which simplified the paper.  The project was sponsored by the National Security Agency under Grant Number H98230-16-1-0048.

\providecommand{\bysame}{\leavevmode\hbox to3em{\hrulefill}\thinspace}
\providecommand{\MR}{\relax\ifhmode\unskip\space\fi MR }
\providecommand{\MRhref}[2]{%
  \href{http://www.ams.org/mathscinet-getitem?mr=#1}{#2}
}
\providecommand{\href}[2]{#2}

\end{document}